\documentclass[twoside]{article}

\usepackage{amsmath,amsthm,amssymb}

\usepackage{newtxtext,newtxmath}

\usepackage[text={12.8cm,19cm},centering,paperwidth=17cm,paperheight=24cm]{geometry}

\usepackage[fontsize=10.4pt]{scrextend}

\def\titlerunning#1{\gdef\titrun{#1}}

\makeatletter
\def\author#1{\gdef\autrun{\def\and{\unskip, }#1}\gdef\@author{#1}}
\def\address#1{{\def\and{\\\hspace*{15.6pt}}\renewcommand{\thefootnote}{}\footnote{#1}}\markboth{\autrun}{\titrun}}
\makeatother

\def\email#1{email: \href{mailto:#1}{#1} }
\def\subjclass#1{\par\bigskip\noindent\textbf{Mathematics Subject Classification 2020.} #1}
\def\keywords#1{\par\smallskip\noindent\textbf{Keywords.} #1}

\newenvironment{dedication}{\itshape\center}{\par\medskip}
\newenvironment{acknowledgments}{\bigskip\small\noindent\textit{Acknowledgments.}}{\par}

\newtheorem{thm}{Theorem}[section]
\newtheorem{cor}[thm]{Corollary}
\newtheorem{lem}[thm]{Lemma}

\theoremstyle{definition}

\newtheorem*{rem}{Remark}

\numberwithin{equation}{section}

\usepackage{dsfont}

\usepackage[hyperfootnotes=false,colorlinks=true,allcolors=blue]{hyperref}

\begin{document}

% Give an abbreviation of the title for the running page headers.
\titlerunning{Semiclassical asymptotics for a class of singular Schr\"odinger operators}

% Here you can enter the full article title.
\title{\textbf{Semiclassical asymptotics for a class of singular Schr\"odinger operators}}

% Here you can enter the full names of authors separated by \and.
\author{Rupert L. Frank \and Simon Larson}

% Please do not enter a date.
\date{}

\maketitle

% Here you can enter the address and email of each author separated by \and by following the example.
\address{R. L. Frank: Mathematisches Institut, Ludwig-Maximilans Universit\"at M\"unchen, Theresienstr.~39, 80333 M\"unchen, Germany, and Department of Mathematics, California Institute of Technology, Pasadena, CA 91125, USA; \email{r.frank@lmu.de, rlfrank@caltech.edu} \and S. Larson: Department of Mathematics, California Institute of Technology, Pasadena, CA 91125, USA; \email{larson@caltech.edu}}

% Here you can enter an optional dedication.
\begin{dedication}
To Ari Laptev on the occasion of his 70th birthday
\end{dedication}

% Here you can enter the abstract, MSC classes, and keywords.
\begin{abstract}
Let $\Omega \subset \mathbb{R}^d$ be bounded with $C^1$ boundary. In this paper we consider Schr\"odinger operators $-\Delta+ W$ on $\Omega$ with $W(x)\approx\textup{dist}(x, \partial\Omega)^{-2}$ as $\textup{dist}(x, \partial\Omega)\to 0$. Under weak assumptions on $W$ we derive a two-term asymptotic formula for the sum of the eigenvalues of such operators.  
\subjclass{Primary 35P20}
\keywords{Schr\"odinger operator, Semiclassical asymptotics.}
\end{abstract}

% ------------------  Introduction  --------------------

\section{Introduction}

In this paper we consider semiclassical asymptotics for a class of Schr\"odinger operators
on bounded sets $\Omega\subset \mathbb{R}^d$ with potentials which are singular at the boundary and subject to Dirichlet boundary conditions. Specifically, for a bounded open set $\Omega\subset\mathbb{R}^d$ with $C^1$ boundary we consider Schr\"odinger operators
\begin{equation}\label{eq: informal op def}
  -\Delta+ W(x) \quad \mbox{with } W(x) \approx \textup{dist}(x, \Omega)^{-2} \mbox{ as } \textup{dist}(x, \partial\Omega)\to 0\,.
\end{equation}
These operators have purely discrete spectrum and our main interest is towards the asymptotic behavior of their eigenvalues. Our main result is a two-term asymptotic formula for the sum of the eigenvalues.

Before we formulate our main result it is necessary to explain more precisely how~\eqref{eq: informal op def} is to be interpreted. We shall assume that our potential decomposes as one part which is in $L^\infty_{\rm loc}(\Omega)$ and has the prescribed singular behavior at the boundary and a part which, in comparison, is well-behaved. To simplify the exposition we write
\begin{equation*}
  H_{\Omega, b, V}(h) = -h^2\Delta+h^2\Bigl(b^2(x)-\frac{1}{4}\Bigr)\frac{1}{\textup{dist}(x, \partial\Omega)^2}+h^2V(x)-1\quad \mbox{for }h>0\,.
\end{equation*}
Technically, the operator $H_{\Omega, b, V}(h)$ is defined through the quadratic form
\begin{equation}\label{eq: H quad form}
  u \mapsto \int_\Omega\Bigl(h^2|\nabla u(x)|^2 + h^2\Bigl(b^2(x)-\frac{1}{4}\Bigr)\frac{|u(x)|^2}{\textup{dist}(x, \partial\Omega)^2}+h^2 V(x)|u(x)|^2-|u(x)|^2\Bigr)\,dx
 \end{equation} 
 with form domain $\{ u\in H_0^1(\Omega): V_+u^2 \in L^1(\Omega)\}$. Throughout we shall assume that $V \in L^1(\Omega)$, $V_- \in L^{1+d/2}(\Omega)$, and that $b \in L^\infty(\Omega)$ is positive and satisfies
\begin{equation}\label{eq: b regularity}
  \lim_{r\to 0^+}\int_{\partial\Omega}\biggl[\,\sup_{y\in B_r(x) \cap \Omega}b(y)-\inf_{y\in B_r(x) \cap \Omega}b(y)\biggr] \,d\mathcal{H}^{d-1}(x) = 0\,.
\end{equation}
Here and in what follows we define $x_\pm = \frac{|x|\pm x}{2}$ and note that with this convention both $x_+$ and $x_-$ are non-negative. As a consequence of Hardy's inequality, the assumptions on $V$ and $b$ ensure that the quadratic form \eqref{eq: H quad form} is bounded from below and closed. Therefore it generates a selfadjoint, bounded from below operator $H_{\Omega,b,V}(h)$ in $L^2(\Omega)$. 

We emphasize that by positivity of $b$, we mean $\inf_\Omega b>0$. This assumption can naturally be relaxed to require positivity only in a neighborhood of the boundary by adjusting $V$ correspondingly. The regularity assumption~\eqref{eq: b regularity} implies that $b|_{\partial\Omega}$ can be made sense of as an element of $L^\infty(\partial\Omega)$; indeed, by~\eqref{eq: b regularity}, $b$ has a well-defined limit $\mathcal{H}^{d-1}$-almost everywhere on $\partial\Omega$ which is finite since $b\in L^\infty(\Omega)$. Our main result can now be stated as follows:

\begin{thm}\label{thm: main result}
  Let $\Omega \subset \mathbb{R}^d$ be open and bounded with $C^1$ boundary, $V \in L^1(\Omega)$ with $V_- \in L^{1+d/2}(\Omega)$, and let $b\in L^\infty(\Omega)$ be positive and satisfy~\eqref{eq: b regularity}. Then, as $h\to 0^+$,
  \begin{equation*}
    \textup{Tr}(H_{\Omega, b, V}(h))_- = L_dh^{-d}|\Omega| - \frac{L_{d-1}}{2}h^{-d+1}\int_{\partial\Omega}b(x)\,d\mathcal{H}^{d-1}(x) + o(h^{-d+1})\,,
  \end{equation*}
  where $L_d = (4\pi)^{-d/2}\Gamma(2+d/2)^{-1}.$
\end{thm}

As a corollary of Theorem~\ref{thm: main result} we deduce:

\begin{cor}\label{cor: convergence of density}
  Let $\Omega\subset \mathbb{R}^d$ be open and bounded with $C^1$ boundary. Then, with $\Delta_\Omega$ denoting the Dirichlet Laplace operator in $\Omega$, as $h \to 0^+$ and in the sense of measures
  \begin{equation*}
    h^{d+1}\frac{\mathds{1}(-h^2\Delta_\Omega\leq 1)(x, x)}{\mathrm{dist}(x, \partial\Omega)^2}\,dx \to \frac{L_{d-1}}{2}\mathcal{H}^{d-1}|_{\partial\Omega}\,.
  \end{equation*}
\end{cor}

\begin{proof}
  The corollary follows from a standard Feynman--Hellmann argument (cf.~\cite{LiebSimon_AdvMath77}) and Theorem~\ref{thm: main result} applied with the potential $W(x) = t f(x)/\mathrm{dist}(x, \partial\Omega)^2$ for $f \in C(\overline{\Omega})$ and sending first $h$ then $t$ to zero.
\end{proof}

Spectral asymptotics for differential operators that degenerate at the boundary of the domain are not new. However, the results in the literature mainly concern cases in which the operator degenerates at leading order and how this affects the first term in the asymptotics, see~\cite{BirmanSolomjak_Survey79,BirmanSolomjak_QuantEmbedding} and references therein. 
While the class of operators considered here is drastically less singular, our interest is towards the effect of the degeneracy on the second term in the asymptotics.

In the special case of the Dirichlet Laplacian, i.e.\ $V\equiv 0$ and $b\equiv1/2$, Theorem~\ref{thm: main result} was proved in~\cite{FrankGeisinger_11,FrankGeisinger_12}. The strategy of our proof follows closely that developed there, but several new obstacles need to be circumvented in the presence of the potential, which is singular at the boundary. The idea is to localize the operator in balls whose size varies depending on the distance to the boundary and $h$. In a ball far from the boundary the influence of the boundary conditions and the potential both have a negligible effect and precise asymptotics can be obtained through standard methods. In a ball close to the boundary the regularity of the boundary allows to map the problem to a half-space where asymptotics are obtained by explicitly diagonalizing an effective operator. The main new ingredients needed here is to control how the straightening of the boundary affects the singular part of the potential and to understand how the potential enters in the half-space problem. 

The works~\cite{FrankGeisinger_11,FrankGeisinger_12} for domains with $C^1$ boundaries were extended to the case of Lipschitz boundaries in \cite{FrankLarson}, see also \cite{FrankLarson_20}. Since the (weak) Hardy constant can be smaller than $1/4$ for Lipschitz domains, it is not clear how to generalize the results of the present paper to this setting.

The plan for the paper is as follows. In Section~\ref{sec: Prelmininaries} we recall a number of results concerning changes of variables mapping $\partial \Omega$ locally to a hyperplane. In particular, Lemma~\ref{lem: A geometric lemma} describes how such a mapping affects the singular part of our potential. We also prove a local Hardy--Lieb--Thirring inequality which will be crucial in controlling error terms appearing in our analysis, and which replaces the Lieb--Thirring inequality in~\cite{FrankGeisinger_11} in the absence of a singular potential. In Section~\ref{sec: Local asymptotics} we provide local asymptotics, both in the bulk of our domain and close to the boundary. Finally, in Section~\ref{sec: local to global} we adapt the localization procedure developed in~\cite{FrankGeisinger_11,FrankGeisinger_12,FrankLarson} to our current setting and use it to piece together the local asymptotics of Section~\ref{sec: Local asymptotics}, thus proving Theorem~\ref{thm: main result}.

The letter $C$ will denote a constant whose value can change at each occurence.

\medskip

We are deeply grateful to Ari Laptev for sharing his fascination for spectral estimates and Hardy's inequality with us and we would like to dedicate this paper to him on the occasion of his 70th birthday.

\section{Preliminaries}\label{sec: Prelmininaries}

\subsection{Straightening the boundary}\label{sec: Straightening of bdry}
Let $\mathbb{R}^d_+ = \{ y\in \mathbb{R}^d: y_d >0\}$. Let $B\subset \mathbb{R}^d$ be an open ball of radius $\ell$ centred at a point $x_0\in \partial\Omega$. By rotating and translating we may assume that $x_0=0$ and $\nu_0=(0, \ldots, 0, 1)$ is the inward pointing unit normal to $\partial \Omega$ at $x_0$. Since $\Omega$ is bounded with $C^1$ boundary, there exists a non-decreasing modulus of continuity $\omega\colon \mathbb{R}_+\to [0, 1]$ such that, if $\ell$ is small enough, there exists a function $f\colon \mathbb{R}^{d-1}\to \mathbb{R}$ satisfying $|\nabla f(x')| \leq \omega(|x'|)$ such that
\begin{equation*}
  \partial\Omega \cap B_{2\ell}(0) = \{(x', x_d) \in \mathbb{R}^{d-1}\times \mathbb{R}: x_d = f(x')\}\cap B_{2\ell}(0)\,.
\end{equation*}
Note that, by the choice of coordinates, $f(0)=0$ and $\nabla f(0)=0$.

Set $\,\mathcal{X}=\{(x', x_d)\in \mathbb{R}^{d-1}\times \mathbb{R}: |x'|<2\ell\}$. We define a diffeomorphism $\Phi \colon \mathcal{X} \to \mathbb{R}^d$
by $\Phi_j(x)=x_j$ for $j=1, \ldots, d-1$ and $\Phi_d(x)= x_d-f(x')$. Note that the Jacobian determinant of $\Phi$ equals $1$ and that the inverse of $\Phi$ is well-defined on $\Phi(\mathcal{X})=\mathcal{X}$. 
The inverse is given by $\Phi_j^{-1}(y)=y_j$ for $j=1, \ldots, d-1$ and $\Phi^{-1}_d(y)=y_d+f(y')$. 

In the following lemma we gather some results whose proofs are standard and can be found, for instance, in~\cite[Section 4]{FrankGeisinger_12}.
\begin{lem}[Straightening of the boundary]\label{lem: Straightening of boundary}
  Let $B, \Phi$ be as above and for $u\colon B \to \mathbb{R}$ set $\tilde u = u \circ \Phi^{-1}$. For $0<\ell \leq c(\omega)$ and with $C$ depending only on $d$, we have:
  \begin{enumerate}
    \item \label{itm: volume preservation} if $u \in L^1(B)$ then
    \begin{equation*}
    \int_{B} u(x)\,dx = \int_{\Phi(B)}\tilde u(y)\,dy\,.
  \end{equation*}

  \item\label{itm: boundary integral change} if $u \in L^\infty(\partial\Omega\cap B)$ then
  \begin{equation*}
    \biggl|\int_{\partial\Omega \cap B}u(x)\,d\mathcal{H}^{d-1}(x)- \int_{\partial\mathbb{R}^{d}_+ \cap \Phi(B)}\tilde u(y)\,d\mathcal{H}^{d-1}(y)\biggr|\leq C \ell^{d-1}\omega(\ell)^2\|u\|_{L^\infty}\,.
  \end{equation*}

  \item\label{itm: gradient estimate} if $u\in H^1_0(\Omega \cap B)$ then $\tilde u \in H^1_0(\mathbb{R}^d_+ \cap \Phi(B))$ and
  \begin{equation*}
    \biggl|\int_{\Omega\cap B} |\nabla u(x)|^2\,dx - \int_{\mathbb{R}^d_+ \cap \Phi(B)} |\nabla \tilde u(y)|^2\,d y \biggr| \leq C \omega(\ell) \int_{\mathbb{R}^d_+ \cap \Phi(B)} |\nabla \tilde u(y)|^2\,dy\,.
  \end{equation*}

  \item\label{itm: sup gradient estimate} if $u \in C^1_0(\mathbb{R}^d)$ is supported in $\overline{B}$ then, after extension by zero, $\tilde u \in C^1_0(\mathbb{R}^d)$ with $\mathrm{supp}\,\tilde u\subseteq \overline{B_{2\ell}(0)}$ and $\|\nabla \tilde u\|_{L^\infty} \leq C\|\nabla u\|_{L^\infty}$.

  \end{enumerate}

\end{lem}

In addition to the properties in Lemma~\ref{lem: Straightening of boundary} we will need the following result which enables us to control the change of the singular part of our potentials:
\begin{lem}\label{lem: A geometric lemma}
  Let $B, \Phi$ be as above. There is a constant $C$ depending only on $d$ such that for any $x \in B \cap \Omega$,
  \begin{equation}\label{eq:distdist}
    0 \leq \frac{1}{\textup{dist}(x, \partial\Omega)^2}- \frac{1}{\textup{dist}(\Phi(x),\partial \mathbb{R}^d_+)^2} \leq C\frac{\omega(2\ell)^2}{\textup{dist}(\Phi(x),\partial \mathbb{R}^d_+)^2}\,.
  \end{equation}
\end{lem}

\begin{proof}
  By definition of $f$, $(x', f(x'))\in \partial\Omega$, thus $\textup{dist}(x, \partial\Omega) \leq |x-(x', f(x'))|=|x_d-f(x')| = \textup{dist}(\Phi(x), \partial \mathbb{R}^d_+)$ which implies the lower bound in \eqref{eq:distdist}.

  To prove the upper bound, let $z= (z', f(z'))\in \partial\Omega$ be such that $\textup{dist}(x, \partial\Omega)=|x- z|$. Since $\partial\Omega$ is parametrized by $f$ in the larger ball $B_{2\ell}(x_0)$ it is clear that such a point exists and that $z\in B_{2\ell}(x_0)$. The point $z$ might not be uniquely determined but that will not play any role in what follows. 

  We begin by rewriting the expression we want to bound in terms of $z$:
  \begin{align*}
    \frac{1}{\textup{dist}(x, \partial\Omega)^2}- \frac{1}{\textup{dist}(\Phi(x),\partial \mathbb{R}^d_+)^2} 
    &=
    \frac{1}{|x-z|^2}- \frac{1}{|x_d-f(x')|^2}\\
     &=
     \frac{(f(x')-f(z'))(f(x')+f(z')-2x_d)-|x'-z'|^2}{|x-z|^2|x_d-f(x')|^2}\,.
  \end{align*}
  Since $f$ is $C^1$ and by the definition of $z$ it holds that
  \begin{equation*}
     x= z+ |x-z| \frac{(-\nabla f(z'), 1)}{\sqrt{1+|\nabla f(z')|^2}}\,.
   \end{equation*} 
   Consequently,
   \begin{equation}\label{eq: comparison distances}
     |x'-z'|^2 = |x-z|^2\frac{|\nabla f(z')|^2}{1+|\nabla f(z')|^2}
   \quad\mbox{and}\quad
     |x_d-f(z')|^2  = \frac{ |x-z|^2}{1+|\nabla f(z')|^2}\,.
   \end{equation}
   Note also that $f(x')\leq f(z')\leq x_d$. From the above identities one finds
   \begin{equation}\label{eq: distances in terms of f}
   \begin{aligned}
    \frac{1}{\textup{dist}(x, \partial\Omega)^2}- \frac{1}{\mathrm{dist}(\Phi(x), \partial \mathbb{R}^d_+)^2} 
     &=
     \frac{1}{|x_d-f(x')|^2}\Biggl[\frac{|f(x')-f(z')|^2}{|x-z|^2} \\
     &\quad
     +2\frac{|f(x')-f(z')|}{|x-z|\sqrt{1+|\nabla f(z')|^2}}- \frac{|\nabla f(z')|^2}{1+|\nabla f(z')|^2}\Biggr]\,.
  \end{aligned}
  \end{equation}
  By the fundamental theorem of calculus and~\eqref{eq: comparison distances}
  \begin{align*}
    |f(x')-f(z')| = \biggl|(x'-z')\int_0^1 \nabla f(t x'+(1-t)z')\,dt\biggr|
    \leq
    \omega(2\ell)^2|x-z| \,.
  \end{align*}
  Therefore
  \begin{align*}
    \frac{|f(x')-f(z')|^2}{|x-z|^2}+ 2 \frac{|f(x')-f(z')|}{|x-z|\sqrt{1+|\nabla f(z')|^2}}
     - \frac{|\nabla f(z')|^2}{1+|\nabla f(z')|^2}
     \leq C \omega(2\ell)^2\,. 
  \end{align*}
  Combined with~\eqref{eq: distances in terms of f} this completes the proof of Lemma~\ref{lem: A geometric lemma}.
\end{proof}

\subsection{A local Hardy--Lieb--Thirring inequality}

The aim of this subsection is to prove a bound for localized traces of our operator. Before stating the result we recall the following Hardy inequality due to Davies~\cite{Davies95} (combine his Theorems 2.3 and 2.4).
\begin{lem}\label{lem: Davies Hardy ineq}
  Let $\Omega \subset \mathbb{R}^d$ be open and bounded with $C^1$-boundary. Then for any $\varepsilon>0$ there is a $c_H(\varepsilon, \Omega)\geq 0$ such that for all $u \in H^1_0(\Omega)$,
  \begin{equation*}
    \int_{\Omega}|\nabla u(x)|^2 \,dx + \Bigl(\varepsilon-\frac{1}{4}\Bigr) \int_\Omega \frac{|u(x)|^2}{\mathrm{dist}(x, \partial\Omega)^2}\,dx \geq - c_H(\varepsilon, \Omega) \int_\Omega |u(x)|^2\,dx\,.
  \end{equation*}
\end{lem}

\begin{rem}
  Lemma~\ref{lem: Davies Hardy ineq} can be proved in a direct manner by using a partition of unity and appealing to Lemmas~\ref{lem: Straightening of boundary} and~\ref{lem: A geometric lemma}. In particular, this allows one to quantify the best constant $c_H$ in terms of the $C^1$-regularity of $\partial\Omega$. Indeed, such a proof yields the bound $c_H(\varepsilon, \Omega) \leq \tfrac{C}{\omega^{-1}(\varepsilon)^2}$ for a constant $C$ depending only on the dimension and $\omega^{-1}$ is the inverse of the $C^1$-modulus of continuity of $\partial \Omega$.
\end{rem}

With Lemma~\ref{lem: Davies Hardy ineq} in hand we move on to the main result of this subsection. Specifically, the following local Hardy--Lieb--Thirring type inequality for $H_{\Omega, b, V}$ (cf.~\cite{FrankLoss_HSM_2012}):

\begin{lem}
\label{lem: local HLT}
  Let $\Omega, b, V$ be as in Theorem~\ref{thm: main result}. Let $\phi \in C^1_0(\mathbb{R}^d)$ be supported in a ball $\overline{B}$ of radius $\ell$ and set $\underline{b} = \inf_{\Omega \cap B} b$. 
  If $0<h\leq K \min\{\ell, c_H(\underline{b}^2/2, \Omega)^{-1/2}\}$, then
  \begin{equation*}
    \textup{Tr}(\phi H_{\Omega, b, V}(h)\phi)_- \leq C \min\{\underline{b}, 1\}^{-d}\ell^dh^{-d} \Bigl(1 + h^2 \|V_-\|^{1+d/2}_{L^{1+d/2}(\Omega \cap B)}\Bigr)\,,
  \end{equation*}
  where the constant $C$ depends only on $d, K,$ and $\|\phi\|_{L^\infty}$. 
\end{lem}

\begin{proof}[Proof of Lemma~\ref{lem: local HLT}]
  By assumption, $\underline{b}>0$.
  By the variational principle and for any $\delta \in (0, 1/2]$, we find 
  \begin{align*}
    \phi H_{\Omega, b, V}(h) \phi &\geq
    \phi\Bigl(-h^2 \delta \Delta  - h^2 V_-(x) -1\\
    &\quad 
    - h^2(1-\delta)\Bigl(-\Delta + (1-\delta)^{-1}\Bigl(\underline{b}^2- \frac{1}{4}\Bigr) \frac{1}{\mathrm{dist}(x, \partial\Omega)^2}\Bigr) \Bigr)\phi\,.
  \end{align*}
  Since $\delta \in (0, 1/2]$ we have
  \begin{equation*}
    (1-\delta)^{-1}\Bigl(\underline{b}^2- \frac{1}{4}\Bigr) \geq (1+2\delta)\Bigl(\underline{b}^2- \frac{1}{4}\Bigr) > \underline{b}^2 - \frac{\delta}{2} - \frac{1}{4}\,.
  \end{equation*}
  Thus, setting $\delta = \min\{\underline{b}^2, 1/2\} \leq 1/2$, Lemma~\ref{lem: Davies Hardy ineq} implies with $c_0 = c_H(\underline{b}^2/2, \Omega)$ that
  \begin{equation}\label{eq: lower bound by Hardy}
    \phi H_{\Omega, b, V}(h)\phi \geq
    \phi(-h^2 \delta \Delta - c_0h^2
     - h^2 V_-(x) -1) \phi\,.
  \end{equation}
  Consequently, for any $0<\rho< 1$, the variational principle and~\eqref{eq: lower bound by Hardy} yields
  \begin{align*}
    \mathrm{Tr}(\phi H_{\Omega, b, V}(h)\phi)_-
    &\leq \mathrm{Tr}(\phi(-h^2\delta(1-\rho)\Delta -c_0h^2-1)\phi)_-\\
    &\quad +
    \mathrm{Tr}(\phi(-h^2\delta\rho\Delta - h^2V_-)\phi)_-\,.
  \end{align*}
  Using the Berezin--Li--Yau inequality
  \begin{align*}
    \mathrm{Tr}(\phi(-h^2\delta(1-\rho)&\Delta -c_0h^2-1)\phi)_- \\
    &\leq C (1+c_0h^2)^{1+d/2}(1-\rho)^{-d/2}\delta^{-d/2}h^{-d}\ell^{d}\,,
  \end{align*}
  with $C>0$ depending on $d$ and $\|\phi\|_{L^\infty}$. For the remaining term the Lieb--Thirring inequality implies
  \begin{equation*}
     \mathrm{Tr}(\phi(-h^2\delta\rho\Delta - h^2V_-)\phi)_- \leq C h^2\delta^{-d/2}\rho^{-d/2}\|V_-\|_{L^{1+d/2}(\Omega \cap B_\ell)}^{1+d/2}\,,
   \end{equation*} 
   for some $C>0$ depending only on $d$. Gathering the estimates and setting $\rho = h^2/(2K^2\ell^2)< 1$ completes the proof.
\end{proof}

\section{Local asymptotics}\label{sec: Local asymptotics}
\subsection{Local asymptotics in the bulk}

\begin{lem}
\label{lem: local asymptotics bulk}
  Let $\phi \in C^1_0(\mathbb{R}^d)$ be supported in a ball $\overline{B}$ of radius $\ell>0$ and satisfy
  \begin{equation}\label{eq: phi grad bound bulk asymptotics}
    \|\nabla \phi\|_{L^\infty(\mathbb{R}^d)}\leq M \ell^{-1}\,.
  \end{equation}
  If $V\in L^1(B)$ is such that $V_- = V_0+V_1$ with $0\leq V_0 \in L^\infty(B)$ and $V_1 \in L^{1+d/2}(B)$ then, for $0<h\leq K \min\{\ell, \|V_0\|_\infty^{-1/2}\}$,
  \begin{align*}
    \Bigl|\textup{Tr}(\phi (-h^2&\Delta+h^2V-1)\phi)_- - L_d h^{-d}\int_B \phi^2(x)\,dx \Bigr|\\
    & \leq
    C h^{-d+2}\Bigl[\ell^{d-2}+\ell^d\|V_0\|_{L^\infty(B)} + \ell^d\|V_1\|^{1+d/2}_{L^{1+d/2}(B)}+\|V_+\|_{L^1(B)}\Bigr]\,,
  \end{align*}
  where the constant $C$ depends only on $d, M, K$.
\end{lem}

\begin{proof}
Throughout the proof we set $H_V = H_{\mathbb R^d,0,V}=-h^2\Delta+ h^2V-1$ in $L^2(\mathbb R^d)$.

To prove the lower bound, consider the operator $\gamma$ with integral kernel
\begin{equation*}
   \gamma(x, y) = \frac{1}{(2\pi)^d}\chi(x)\int_{|\xi|< h^{-1}}e^{i \xi(x-y)}\,d\xi\,\chi(y)\,,
\end{equation*} 
where $\chi \in C_0^\infty(\mathbb{R}^d)$ with $0\leq \chi \leq 1$ and $\chi \equiv 1$ on $B$. The operator $\gamma$ is trace class and satisfies $0\leq \gamma \leq \mathbf{1}$. Therefore, the variational principle implies that
\begin{align*}
  \textup{Tr}(\phi H_V \phi)_-
  &\geq
  \textup{Tr}(\phi H_{V_+}\phi)_-\\
  &\geq
  -\textup{Tr}(\gamma \phi H_{V_+}\phi)\\
  &=
  - \frac{1}{(2\pi)^d}\int_{|\xi|< h^{-1}} \Bigl(h^2 \|\nabla e^{i\xi\, \cdot \,}\phi\|_{L^2(\mathbb{R}^d)}^2+h^2\|V_+ \phi^2\|_{L^1(\mathbb{R}^d)}-\|\phi\|_{L^2(\mathbb{R}^d)}^2\Bigr)\,d\xi\\
  &=
  L_d h^{-d}\int_B \phi^2(x)\,dx
  - C h^{-d+2}\Bigl(\|\nabla \phi\|_{L^2(\mathbb{R}^d)}^2+\|V_+ \phi^2\|_{L^1(\mathbb{R}^d)}\Bigr)\,.
\end{align*}
Since, by~\eqref{eq: phi grad bound bulk asymptotics}, $\|\phi\|_{L^\infty}\leq M$ and $\|\nabla \phi\|_{L^2(\mathbb{R}^d)}^2 \leq C \ell^{d-2}$ this proves the lower bound.

It remains to prove the upper bound.
For any $0<\rho \leq 1/2$ 
\begin{align*}
  \textup{Tr}(\phi H_{V}\phi)_- 
  &\leq \textup{Tr}(\phi H_{V_-}\phi)_-\\
  &\leq 
  \textup{Tr}(\phi (-h^2(1-\rho)\Delta-h^2V_0-1)\phi)_- + h^2\textup{Tr}(\phi(-\rho\Delta-V_1)\phi)_-\,.
\end{align*}
To bound the second term we apply the Lieb--Thirring inequality to conclude that
\begin{align*}
  h^2\textup{Tr}(\phi(-\rho\Delta-V_1)\phi)_-
  &\leq
  h^2\textup{Tr}(\phi(-\rho\Delta-V_1\mathrm{1}_B)_-\phi)
  \leq
  Ch^2\rho^{-d/2} \int_{B}|V_1(x)|^{1+d/2}\,dx\,,
\end{align*}
where we again used $\|\phi\|_{L^\infty}\leq M$. Since $V_0\in L^\infty(B)$, we can bound
\begin{align*}
  \textup{Tr}(\phi (-h^2(1-\rho)\Delta-h^2V_0-1)\phi)_-
  &\leq
  \textup{Tr}(\phi(-h^2(1-\rho)\Delta - h^2 \sup_{B}V_0-1)\phi)_-\\
  &=
  (1+h^2\sup_B V_0) \textup{Tr}(\phi (-\tilde h^2\Delta-1)\phi)_-
\end{align*}
with $\tilde h = h(1-\rho)^{1/2}(1+h^2\sup_B V_0)^{-1/2}$.
By the Berezin--Li--Yau inequality,
\begin{equation*}
  \textup{Tr}(\phi (-\tilde h^2\Delta-1)\phi)_-\leq  L_d \tilde h^{-d}\int_B \phi^2(x)\,dx\,.
\end{equation*}
Combining the above we have arrived at
\begin{align*}
  \textup{Tr}(\phi H_{V}\phi)_-
  &\leq
  L_d h^{-d}\int_B \phi^2(x)\,dx +
  C h^2 \rho^{-d/2}\int_B |V_1(x)|^{1+d/2}\,dx\\
  &\quad +
  L_d h^{-d}\Bigl[(1-\rho)^{-d/2}(1+h^2\sup_B V_0)^{1+d/2}-1\Bigr]\int_B \phi^2(x)\,dx \\
  &\leq 
  L_d h^{-d}\int_B \phi^2(x)\,dx
  +
  C h^2 \rho^{-d/2}\int_B |V_1(x)|^{1+d/2}\,dx\\
  &\quad +
  C h^{-d}\Bigl[\rho+h^2\sup_B V_0\Bigr]\int_B \phi^2(x)\,dx\,,
\end{align*}
where $C$ depend only on $d, K, M$. Setting $\rho = h^2/(2K^2\ell^2)\leq 1/2$ and using $\int \phi^2 \leq  C \ell^d$ completes the proof.
\end{proof}

\subsection{Local asymptotics near the boundary}

In this section we prove the following local asymptotic expansion close to the boundary:
\begin{thm}\label{thm: local asymptotics boundary}
	Let $\Omega, b, V$ be as in Theorem~\ref{thm: main result}. Let $\phi \in C^1_0(\mathbb{R}^d)$ be supported in a ball $\overline{B}$ of radius $\ell$ and satisfy
  \begin{equation*}
    \|\nabla \phi\|_{L^\infty(\mathbb{R}^d)}\leq M \ell^{-1}\,.
  \end{equation*} 
  Assume that $\textup{dist}(B, \partial\Omega)\leq 2\ell$, and set $\underline{b} = \inf_{B\cap \Omega} b$. For $0<\ell\leq c(\Omega, \underline{b})$ and $0<h\leq K\ell$,
  \begin{align*} 
    \biggl|\textup{Tr}&(\phi H_{\Omega, b, V}(h)\phi)_- - L_dh^{-d}\int_\Omega \phi^2(x)\,dx + \frac{L_{d-1}}{2}h^{-d+1}\int_{\partial\Omega}\phi^2(x)b(x)\,d\mathcal{H}^{d-1}(x)\biggr|\\
    &\leq  \ell^dh^{-d}o_{\ell\to 0^+}(1) 
    + O(h^{-d+1})\int_{\partial\Omega}\phi^2(x)\biggl[\sup_{y\in B_{2\ell}(x)}b(y)-\inf_{y\in B_{2\ell}(x)}b(y)\biggr]\,d\mathcal{H}^{d-1}(x)\\
    &\quad
    + O(h^{-d+2})\Bigl(\ell^{d-2}\log(\ell/h)+\|V_+\|_{L^1(\Omega \cap B)}+\ell^d \|V_-\|^{1+d/2}_{L^{1+d/2}(\Omega \cap B)}\Bigr)
  \end{align*}
  Moreover, the error terms and the implicit constants can be quantified in terms of the $C^1$-regularity of $\partial\Omega$ and $M, K, \|b\|_{L^\infty(\Omega\cap B)}, \underline{b}$.
\end{thm}
The proof of Theorem~\ref{thm: local asymptotics boundary} will be split into several lemmas. The first of which reduces our problem to the corresponding in a half-space:

\begin{lem}\label{lem: reduction to halfspace}
  Let $\Omega, b , V$ be as in Theorem~\ref{thm: main result}. Let $\phi \in C^1_0(\mathbb{R}^d)$ be supported in a ball $\overline{B}$ of radius $\ell$ such that $\textup{dist}(B, \partial\Omega)\leq 2\ell$, and $\inf_{B\cap \Omega} b = \underline{b}>0$. For $0<\ell\leq c(\Omega, \underline{b})$ and $0<h\leq K\ell$ with $\tilde \phi = \phi\circ \Phi^{-1}, \tilde V = V \circ \Phi^{-1}$,
  \begin{align*}
    \textup{Tr}(\tilde\phi H_{\mathbb{R}^d_+, \overline{b}, \tilde V}&(h)\tilde\phi)_- -  \ell^dh^{-d}o_{\ell\to 0^+}(1)\Bigl(1 + h^2\|V_-\|^{1+d/2}_{L^{1+d/2}(\Omega \cap B)}\Bigr) \\
    &\leq \textup{Tr}(\phi H_{\Omega, b, V}(h)\phi)_-\\
    &\leq \textup{Tr}(\tilde\phi H_{\mathbb{R}^d_+, \underline{b}, \tilde V}(h)\tilde\phi)_- +  \ell^dh^{-d}o_{\ell\to 0^+}(1)\Bigl(1 + h^2\|V_-\|^{1+d/2}_{L^{1+d/2}(\Omega \cap B)}\Bigr)
  \end{align*}
  where $\overline{b} = \sup_{x\in B\cap \Omega}b(x)$. Moreover, the error terms and the implicit constants can be quantified in terms of the $C^1$-regularity of $\partial\Omega$ and $K, \underline{b}, \overline{b}, \|\phi\|_{L^\infty}$.
\end{lem}

\begin{proof}
  Provided $\ell$ is small enough there exists a ball $B'\supset B$ with centre on $\partial \Omega$ and radius $4\ell$ which satisfies the assumptions in Section~\ref{sec: Straightening of bdry}. Let $\Phi$ be the associated diffeomorphism.
  
  We split the proof into two parts, in the first part we prove the upper bound and in the second we prove the lower bound.
  \smallskip

  {\noindent\bf Part 1:} (Proof of the upper bound) By the variational principle
  \begin{equation*}
    \textup{Tr}(\phi H_{\Omega, b, V}(h)\phi)_- \leq \textup{Tr}(\phi H_{\Omega, \underline{b}, V}(h)\phi)_-\,.
  \end{equation*}
  Moreover, by Lemma~\ref{lem: Straightening of boundary} there exists $C_0>0$ depending only on $d$ such that
  \begin{equation*}
    \textup{Tr}(\phi H_{\Omega, \underline{b}, V}(h)\phi)_- \leq \textup{Tr}\Bigl(\tilde \phi \Bigl(-h^2(1-C_0\omega(4\ell))\Delta_{\mathbb{R}^d_+}+ h^2\frac{\underline{b}^2-1/4}{\textup{dist}(\Phi^{-1}(\,\cdot\,), \partial\Omega)^2}+h^2 \tilde V-1\Bigr)\tilde \phi\Bigr)_-\,.
  \end{equation*}
  We claim that 
  \begin{align*}
    \textup{Tr}\Bigl(&\tilde \phi \Bigl(-h^2(1-C_0\omega(4\ell))\Delta_{\mathbb{R}^d_+}+ h^2\frac{\underline{b}^2-1/4}{\textup{dist}(\Phi^{-1}(\,\cdot\,), \partial\Omega)^2}+h^2 \tilde V-1\Bigr)\tilde \phi\Bigr)_-\\
    &\leq
    \textup{Tr}\Bigl(\tilde \phi \Bigl(-h^2(1-C_0\omega(4\ell))\Delta_{\mathbb{R}^d_+}+ h^2\frac{\underline{b}^2-1/4-C\omega(8\ell)^2}{\textup{dist}(\,\cdot\,, \partial\mathbb{R}^d_+)^2}+h^2 \tilde V-1\Bigr)\tilde \phi\Bigr)_-
  \end{align*}
  for a constant $C$ depending only on $d$. Indeed, if $\underline{b}\geq 1/2$ Lemma~\ref{lem: A geometric lemma} and the variational principle implies
  \begin{align*}
    \textup{Tr}\Bigl(&\tilde \phi \Bigl(-h^2(1-C_0\omega(4\ell))\Delta_{\mathbb{R}^d_+}+ h^2\frac{\underline{b}^2-1/4}{\textup{dist}(\Phi^{-1}(\,\cdot\,), \partial\Omega)^2}+h^2 \tilde V-1\Bigr)\tilde \phi\Bigr)_-\\
    &\leq
    \textup{Tr}\Bigl(\tilde \phi \Bigl(-h^2(1-C_0\omega(4\ell))\Delta_{\mathbb{R}^d_+}+ h^2\frac{\underline{b}^2-1/4}{\textup{dist}(\,\cdot\,, \partial\mathbb{R}^d_+)^2}+h^2 \tilde V-1\Bigr)\tilde \phi\Bigr)_-\\
    &\leq
    \textup{Tr}\Bigl(\tilde \phi \Bigl(-h^2(1-C_0\omega(4\ell))\Delta_{\mathbb{R}^d_+}+ h^2\frac{\underline{b}^2-1/4-C\omega(8\ell)^2}{\textup{dist}(\,\cdot\,, \partial\mathbb{R}^d_+)^2}+h^2 \tilde V-1\Bigr)\tilde \phi\Bigr)_-\,.
  \end{align*}
  Similarly, if $0< \underline{b}< 1/2$ Lemma~\ref{lem: A geometric lemma} and the variational principle implies
  \begin{align*}
    \textup{Tr}\Bigl(&\tilde \phi \Bigl(-h^2(1-C_0\omega(4\ell))\Delta_{\mathbb{R}^d_+}+ h^2\frac{\underline{b}^2-1/4}{\textup{dist}(\Phi^{-1}(\,\cdot\,), \partial\Omega)^2}+h^2 \tilde V-1\Bigr)\tilde \phi\Bigr)_-\\
    &\leq
    \textup{Tr}\Bigl(\tilde \phi \Bigl(-h^2(1-C_0\omega(4\ell))\Delta_{\mathbb{R}^d_+}+ h^2\frac{(\underline{b}^2-1/4)(1+C\omega(8\ell)^2)}{\textup{dist}(\,\cdot\,, \partial\mathbb{R}^d_+)^2}+h^2 \tilde V-1\Bigr)\tilde \phi\Bigr)_-\\
    &\leq
    \textup{Tr}\Bigl(\tilde \phi \Bigl(-h^2(1-C_0\omega(4\ell))\Delta_{\mathbb{R}^d_+}+ h^2\frac{\underline{b}^2-1/4-C\omega(8\ell)^2}{\textup{dist}(\,\cdot\,, \partial\mathbb{R}^d_+)^2}+h^2 \tilde V-1\Bigr)\tilde \phi\Bigr)_-\,.
  \end{align*}

  For any $2C_0\omega(4\ell)< \rho \leq 1/2$ we estimate
  \begin{align*}
    \textup{Tr}\Bigl(&\tilde \phi \Bigl(-h^2(1-C_0\omega(4\ell))\Delta_{\mathbb{R}^d_+}+ h^2\frac{\underline{b}^2-1/4-C\omega(8\ell)^2}{\textup{dist}(\,\cdot\,, \partial\mathbb{R}^d_+)^2}+h^2 \tilde V-1\Bigr)\tilde \phi\Bigr)_-\\
    &\leq
    \textup{Tr}(\tilde \phi H_{\mathbb{R}^d_+, \underline{b}, \tilde V}(h)\tilde \phi)_-\\
    &\quad 
    +
    \textup{Tr}\Bigl(\tilde \phi \Bigl(-h^2(\rho-C_0\omega(4\ell))\Delta_{\mathbb{R}^d_+}+ h^2\frac{\rho(\underline{b}^2-1/4)-C\omega(8\ell)^2}{\textup{dist}(\,\cdot\,, \partial\mathbb{R}^d_+)^2}+h^2\rho\tilde V -\rho\Bigr)\tilde \phi\Bigr)_-\,.
  \end{align*}
  Provided
  \begin{equation}\label{eq: HLT smallness half-space reduction upper bound}
    \frac{\rho(\underline{b}^2-1/4)-C\omega(8\ell)^2}{\rho-C_0\omega(4\ell)} = \Bigl(\underline{b}^2-1/4\Bigr)\frac{1}{1-C_0\omega(4\ell)\rho^{-1}} - C \frac{\omega(8\ell)^2}{\rho-C_0\omega(4\ell)}> -\frac{1}{4}\,,
  \end{equation}
  we can apply the local Hardy--Lieb--Thirring inequality of Lemma~\ref{lem: local HLT} in $\mathbb{R}^d_+$ to bound
  \begin{align*}
    &\textup{Tr}\Bigl(\tilde \phi \Bigl(-h^2(\rho-C_0\omega(4\ell))\Delta_{\mathbb{R}^d_+}+ h^2\frac{\rho(\underline{b}^2-1/4)-C\omega(8\ell)^2}{\textup{dist}(\,\cdot\,, \partial\mathbb{R}^d_+)^2}+h^2\rho \tilde V-\rho\Bigr)\tilde \phi\Bigr)_-\\
    &\ \leq
    C \rho^{1+d/2} \ell^d h^{-d} (\rho-C_0\omega(4\ell))^{-d/2}\Bigr(1+h^2\rho^{d/2}(\rho-C_0\omega(4\ell))^{-d/2}\|V_-\|^{1+d/2}_{L^{1+d/2}(\Omega \cap B)}\Bigr)\\
    &\ \leq
    C \rho \ell^d h^{-d}\Bigr(1+h^2\|V_-\|^{1+d/2}_{L^{1+d/2}(\Omega \cap B)}\Bigr)\,.
  \end{align*}
  Set $\rho = \sqrt{\omega(4\ell)}+ \omega(8\ell)$. Then $\rho>2C_0\omega(4\ell)$ and~\eqref{eq: HLT smallness half-space reduction upper bound} are valid provided $\ell$ is small enough. Therefore, upon collecting  the estimates above we arrive at the bound
  \begin{align*}
    \textup{Tr}(\phi H_{\Omega, b, V}(h)\phi)_- &\leq \textup{Tr}(\tilde\phi H_{\mathbb{R}^d_+, \underline{b}, \tilde V}(h)\tilde\phi)_- \\
    &\quad + C\ell^d h^{-d}\Bigl(\sqrt{\omega(4\ell)}+ \omega(8\ell)\Bigr)\Bigl(1+ h^2\|V_-\|^{1+d/2}_{L^{1+d/2}(\Omega \cap B)}\Bigr)\,,
  \end{align*}
  thus completing the proof of the upper bound.
  \smallskip

  {\noindent \bf Part 2:} (Proof of the lower bound) The proof of the lower bound proceeds as the upper bound but with the roles of $\Omega$ and $\mathbb{R}^d_+$ exchanged.

  By Lemma~\ref{lem: Straightening of boundary},
  \begin{equation*}
    \textup{Tr}(\tilde \phi H_{\mathbb{R}^d_+, \overline{b}, \tilde V}(h)\tilde \phi)_- \!\leq \textup{Tr}\Bigl(\phi \Bigl(-h^2(1+C_0\omega(4\ell))^{-1}\!\Delta_{\Omega}+ h^2\frac{\overline{b}^2-1/4}{\textup{dist}(\Phi(\,\cdot\,), \partial\mathbb{R}^d_+)^2}+h^2 V-1\Bigr)\phi\Bigr)_-.
  \end{equation*}
  If $\ell$ is sufficiently small so that $C_0\omega(4\ell)\leq 1/2$ then $(1+C_0\omega(4\ell))^{-1}\geq 1-C_0\omega(4\ell)>0$, and hence
  \begin{equation*}
    \textup{Tr}(\tilde \phi H_{\mathbb{R}^d_+, \overline{b}, \tilde V}(h)\tilde \phi)_- \leq \textup{Tr}\Bigl(\phi \Bigl(-h^2(1-C_0\omega(4\ell))\Delta_{\Omega}+ h^2\frac{\overline{b}^2-1/4}{\textup{dist}(\Phi(\,\cdot\,), \partial\mathbb{R}^d_+)^2}+h^2 V-1\Bigr)\phi\Bigr)_-\,.
  \end{equation*}
  By splitting into cases depending on the sign of $\overline{b}^2-1/4$ as in the proof of the upper bound one finds 
  \begin{align*}
    \textup{Tr}\Bigl(&\phi \Bigl(-h^2(1-C_0\omega(4\ell))\Delta_{\Omega} +h^2\frac{\overline{b}^2-1/4}{\textup{dist}(\Phi(\,\cdot\,), \partial\mathbb{R}^d_+)^2}+h^2 V-1\Bigr) \phi\Bigr)_-\\
    &\leq
    \textup{Tr}\Bigl(\phi \Bigl(-h^2(1-C_0\omega(4\ell))\Delta_{\Omega}+ h^2\frac{\overline{b}^2-1/4-C\omega(8\ell)^2}{\textup{dist}(\,\cdot\,, \partial\Omega)^2}+h^2 V-1\Bigr)\phi\Bigr)_-
  \end{align*}
  for a constant $C$ depending on $d, \overline{b}$.

For any $2C_0\omega(4\ell)< \rho \leq 1/2$ we estimate
  \begin{align*}
    \textup{Tr}\Bigl(&\phi \Bigl(-h^2(1-C_0\omega(4\ell))\Delta_{\Omega}+ h^2\frac{\overline{b}^2-1/4-C\omega(8\ell)^2}{\textup{dist}(\,\cdot\,, \partial\Omega)^2}+h^2 V-1\Bigr)\phi\Bigr)_-\\
    &\leq
    \textup{Tr}(\phi H_{\Omega, \overline{b}, V}(h)\phi)_-\\
    &\quad 
    +
    \textup{Tr}\Bigl(\phi \Bigl(-h^2(\rho -C_0\omega(4\ell))\Delta_{\Omega}+ h^2\frac{\rho (\overline{b}^2-1/4)-C\omega(8\ell)^2}{\textup{dist}(\,\cdot\,, \partial\Omega)^2}+h^2\rho V-\rho \Bigr)\phi\Bigr)_-\\
    &\leq
    \textup{Tr}(\phi H_{\Omega, b, V}(h)\phi)_-\\
    &\quad 
    +
    \textup{Tr}\Bigl(\phi \Bigl(-h^2(\rho -C_0\omega(4\ell))\Delta_{\Omega}+ h^2\frac{\rho (\overline{b}^2-1/4)-C\omega(8\ell)^2}{\textup{dist}(\,\cdot\,, \partial\Omega)^2}+h^2\rho V-\rho \Bigr)\phi\Bigr)_-\,.
  \end{align*}
Provided the analogue of~\eqref{eq: HLT smallness half-space reduction upper bound} with $\overline{b}$ instead of $\underline{b}$ holds we can apply the local Hardy--Lieb--Thirring inequality of Lemma~\ref{lem: local HLT} to bound
  \begin{align*}
    \textup{Tr}\Bigl(&\phi \Bigl(-h^2(\rho -C_0\omega(4\ell))\Delta_{\Omega}+ h^2\frac{\rho (\overline{b}^2-1/4)-C\omega(8\ell)^2}{\textup{dist}(\,\cdot\,, \partial\Omega)^2}+h^2\rho V-\rho \Bigr)\phi\Bigr)_-\\
    &\leq
    C \rho \ell^dh^{-d}\Bigl(1+ h^2\|V_-\|^{1+d/2}_{L^{1+d/2}(\Omega \cap B)}\Bigr)\,.
  \end{align*}
  Again we can set $\rho = \sqrt{\omega(4\ell)}+ \omega(8\ell)$ and combine the above estimates to arrive at
  \begin{align*}
    \textup{Tr}(\tilde\phi H_{\mathbb{R}^d_+, \overline{b}, \tilde V}(h)\tilde \phi)_- &\leq \textup{Tr}(\phi H_{\Omega, b, V}(h)\phi)_- \\
    &\quad + C\ell^d h^{-d}\Bigl(\sqrt{\omega(4\ell)}+ \omega(8\ell)\Bigr)\Bigl(1+ h^2\|V_-\|^{1+d/2}_{L^{1+d/2}(\Omega \cap B)}\Bigr)\,.
  \end{align*}
  This completes the proof of the lower bound and hence the proof of Lemma~\ref{lem: reduction to halfspace}.
\end{proof}

The proof of Theorem~\ref{thm: local asymptotics boundary} has been reduced to understanding the asymptotics of $\textup{Tr}(\phi H_{\mathbb{R}^d_+, b, V}(h)\phi)_-$ with $b(x)\equiv b>0$.

\begin{lem}\label{lem: halfspace asymptotics}
  Let $\Omega, V$ be as in Theorem~\ref{thm: main result}. Let $\phi \in C^1_0(\mathbb{R}^d)$ be supported in a ball $\overline{B}$ of radius $\ell$ and satisfy 
  \begin{equation}\label{eq: grad bound halfspace asymptotics}
    \|\nabla \phi\|_{L^\infty}\leq M \ell^{-1}\,.
  \end{equation}
  With $b(x) \equiv b> 0$ we have, for $0<h\leq K\ell$,
  \begin{align*}
    \biggl|\textup{Tr}(\phi H_{\mathbb{R}^d_+, b, V}&(h)\phi)_- - L_d h^{-d} \int_{\mathbb{R}^d_+}\phi^2(y)\,dy + \frac{b\,L_{d-1}}{2}h^{-d+1}\int_{\partial \mathbb{R}^d_+} \phi^2(y)\,d\mathcal{H}^{d-1}(y)\biggr|\\
    & \leq C h^{-d+2}\Bigl(\ell^{d-2}|{\log(\ell/h)}|+\|V_+\|_{L^1(\mathbb{R}^d_+ \cap B)}+ \ell^d\|V_-\|^{1+d/2}_{L^{1+d/2}(\mathbb{R}^d_+ \cap B)}\Bigr)\,,
  \end{align*}
  where $C$ depends only on $d, M, K, b$ and can be uniformly bounded for $b$ in compact subsets of $[0, \infty)$.
\end{lem}

\begin{proof}
Our proof proceeds by diagonalizing the operator $H_{\mathbb{R}^d_+, b, 0}(h)$. For the general background on what follows, see~\cite[Chapter XIII]{DunfordSchwartz_II}. 

For $f\in C^2(\mathbb{R}_+)$ define the differential expression
\begin{equation*}
  L_b f(x) = f''(x) - \Bigl(b^2-\frac{1}{4}\Bigr)\frac{f(x)}{x^2}\,.
\end{equation*}
The operator $H_{\mathbb{R}^d_+, b, 0}(h)$ can then be decomposed as
\begin{equation*}
  H_{\mathbb{R}^d_+,b,0}(h) = -h^2\Delta' - h^2L_b\,,
\end{equation*}
where $\Delta' = \sum_{j=1}^{d-1}\frac{\partial^2}{\partial  y_j^2}$ and $L_b$ acts in the $y_d$-coordinate.

For $b > 0, \mu \geq 0$ the ODE
\begin{equation*}
  -L_b u(x)=\mu u(x)
\end{equation*}
has two linearly independent solutions
\begin{align*}
  \psi_{b, \mu}(x)= x^{1/2}J_{b}(x\sqrt{\mu})\quad \mbox{and}\quad \eta_{b, \mu}(x)=x^{1/2}Y_{b}(x\sqrt{\mu})\,.
\end{align*}
If $b\geq 1/2$ only $\psi$ vanishes at $x=0$ while for $b \in (0, 1/2)$ both solutions vanish, indeed $\psi \sim x^{1/2+b}$ and $\eta\sim x^{1/2-b}$ as $x\to 0^+$. However, for any $b\neq \frac{1}{2}$ only the first solution $\psi_{b, \nu}$ is in $H^1$ around zero. In particular, our effective operator $H_{\mathbb{R}^d_+, b, 0}(h)$ is diagonalized through a Fourier transform with respect to $y'$ and a Hankel transform $\mathfrak{H}_{b}$ with respect to $y_d$. Recall that the Hankel transform $\mathfrak{H}_\alpha\colon L^2(\mathbb{R}_+)\to L^2(\mathbb{R}_+)$ is initially defined by
\begin{equation*}
  \mathfrak{H}_\alpha(g)(s)= \int_0^\infty g(t)J_\alpha(s t)\sqrt{st}\,dt \quad \mbox{for }g\in L^1(\mathbb{R}_+)
\end{equation*}
and extended to $L^2(\mathbb{R}_+)$ in a similar manner as the Fourier transform. Moreover, $\mathfrak{H}_\alpha$ is unitary, is its own inverse $\mathfrak{H}_\alpha^2=\mathrm{1}$. Moreover, for $G\in L^\infty(\mathbb{R}_+)$ with compact support and $f\in H_0^1(\mathbb{R}_+)\cap H^2(\mathbb{R}_+)$
\begin{equation*}
 \langle f, G(-L_b) f\rangle_{L^2(\mathbb{R}_+)} = \int_0^\infty G(s^2)|\mathfrak{H}_{b}(f)(s)|^2\,ds\,.
\end{equation*}

By a similar argument as in the proof of Lemma \ref{lem: local HLT} the upper bound can be reduced to the case $V\equiv 0$. Indeed, for any $0<\rho\leq 1/2$,
\begin{align*}
  \textup{Tr}(\phi& H_{\mathbb{R}^d_+, b, V}(h)\phi)_- \\
  &\leq 
  \textup{Tr}(\phi H_{\mathbb{R}^d_+, b, 0}(h(1-\rho))\phi)_- 
  + 
  \textup{Tr}\Bigl(\phi\Bigr(h^2\rho\Delta_{\mathbb{R}^d_+}+ h^2\rho \frac{b^2-1/4}{\textup{dist}(\,\cdot\,, \partial\mathbb{R}^d_+)^2}-h^2V\Bigr)\phi\Bigr)_- \\
  &\leq 
  \textup{Tr}(\phi H_{\mathbb{R}^d_+, b, 0}(h(1-\rho))\phi)_- + C h^2\rho^{-d/2}\|V_-\|^{1+d/2}_{L^{1+d/2}(\mathbb{R}^d_+ \cap B)}\,.
\end{align*}
Set $\rho = h^2/(2K^2\ell^2)$ so that $h^2\rho^{-d/2}=O(\ell^dh^{-d+2})$ and $(h(1-\rho))^{-\beta} = h^{-\beta}(1 + O(\ell^{-2}h^2))$.
The claimed upper bound now follows from the case $V\equiv 0$. 

Using the inequality $\textup{Tr}(\phi H \phi)_- \leq \textup{Tr}(\phi H_- \phi)$, applying the Fourier transform with respect to $y'$ and the Hankel transform in the $y_d$-direction yields
\begin{equation}\label{eq: Upper tracebound halfspace}
\begin{aligned}
  \textup{Tr}(\phi H_{\mathbb{R}^d_+, b, 0}(h)\phi)_-
  &\leq
  \textup{Tr}(\phi (H_{\mathbb{R}^d_+, b, 0}(h))_-\phi)\\
  &=
  \frac{1}{(2\pi)^{d-1}}\iint_{\mathbb{R}^d_+\times \mathbb{R}^d_+} \phi^2(y)(h^2|\xi|^2-1)_- \xi_d y_d J_{b}(\xi_d y_d)^2\,d\xi dy\,.
\end{aligned}
\end{equation}

For the lower bound define the operator $\gamma$ with integral kernel
\begin{equation*}
  \gamma(x, y) = \frac{1}{(2\pi)^{d-1}}\chi(x) \int_{\mathbb{R}^d_+\cap B_{h^{-1}}(0)} e^{i \xi'(x'-y')}\sqrt{\xi_d x_d}J_{b}(\xi_d x_d) \sqrt{\xi_d y_d}J_{b}(\xi_d y_d) \,d \xi\,  \chi(y)\,,
\end{equation*}
where $\chi\in C_0^\infty(\mathbb{R}^d)$ is such that $0\leq \chi \leq 1$ and $\chi\equiv 1$ on $\textup{supp}\, \phi$. The operator $\gamma$ is trace class, satisfies $0\leq \gamma \leq \mathbf{1}$, and its range is contained in the domain of $H_{\mathbb{R}^d_+, b, V}$. Thus, by the variational principle,
\begin{equation}\label{eq: lower trace bound halfspace}
\begin{aligned}
  -\textup{Tr}(\phi &H_{\mathbb{R}^d_+, b, V}(h)\phi)_-\\
  &\leq
  \textup{Tr}(\gamma \phi H_{\mathbb{R}^d_+, b, V_+}(h)\phi)\\
  &=
  \frac{1}{(2\pi)^{d-1}}\iint_{\mathbb{R}^d_+\times\mathbb{R}^d_+}(h^2|\xi|^2 -1)_- \phi^2(x)\xi_d x_dJ_{b}(\xi_dx_d)^2 d\xi dx \\
  &\quad
  + h^{-d+2}\int_{\mathbb{R}^d_+}(V_+(x)\phi^2(x)+|\nabla\phi(x)|^2) \int_0^{1}(x_d th^{-1})J_{b}(x_d th^{-1})^2\,dt dx \\
  &\leq
  \frac{1}{(2\pi)^{d-1}}\iint_{\mathbb{R}^d_+\times\mathbb{R}^d_+}(h^2|\xi|^2 -1)_- \phi^2(x)\xi_d x_dJ_{b}(\xi_dx_d)^2 d\xi dx\\
  &\quad
  + Ch^{-d+2}\int_{\mathbb{R}^d_+}(V_+(x)\phi^2(x)+|\nabla\phi(x)|^2)\,dx\,, 
\end{aligned}
\end{equation}
with $C$ uniformly bounded for $b$ in compact subsets of $[0, \infty)$, since $\|\sqrt{\,\cdot\,}J_b\|_{L^\infty(\mathbb{R}_+)}<\infty$ uniformly for $b$ in compact subsets of $[0, \infty)$ (see~\cite[Chapter 7]{Watson_BesselFunctions}). By~\eqref{eq: grad bound halfspace asymptotics} we can estimate $\|\phi\|_{L^\infty}\leq M$ and $\int_{\mathbb{R}^d_+}|\nabla \phi(x)|^2\,dx \leq C \ell^{d-2}.$

What remains is to understand the common integral in~\eqref{eq: Upper tracebound halfspace} and~\eqref{eq: lower trace bound halfspace}. We begin by extracting the desired leading term:
\begin{align}
  &\frac{1}{(2\pi)^{d-1}}\iint_{\mathbb{R}^d_+\times \mathbb{R}^d_+} \phi^2(y)(h^2|\xi|^2-1)_- \xi_d y_d J_{b}(\xi_d y_d)^2\,d\xi dy\nonumber \\
  &\ =
  L_dh^{-d}\int_{\mathbb{R}^d_+} \phi^2(y)dy \label{eq: extraction of main term}\\
  &\ \quad
  -L_{d-1}h^{-d+1}\int_0^\infty\hspace{-5pt}\int_{\mathbb{R}^{d-1}} \phi^2(y', h t)\,dy'\int_0^1(1-\xi_d^2)^{(d+1)/2} \Bigl(\frac{1}{\pi}-\xi_d t J_{b}(\xi_d t)^2\Bigr)\,d\xi_d dt\,. \nonumber
\end{align}

Define, for $b \geq 0$ and $t\geq 0$,
\begin{equation*}
  P_b(t) = \int_0^1(1-\xi^2)^{(d+1)/2} \Bigl(\frac{1}{\pi}-\xi t J_{b}(\xi t)^2\Bigr)\,d\xi\,.
\end{equation*}
In Lemmas~\ref{lem: Pnu asymptotics} and~\ref{lem: Pnu integral identity} we shall prove that
\begin{equation}\label{eq: P properties}
   \int_0^\infty P_b(t)\, dt = \frac{b}{2}\quad \mbox{and} \quad P_b(t)= O(t^{-2}) \mbox{ as } t\to \infty\,,
\end{equation}
with the implicit constant uniformly bounded for $b$ in compact subsets of $[0, \infty)$.

Using~\eqref{eq: P properties} we can estimate 
\begin{align*}
  \int_0^\infty\int_{\mathbb{R}^{d-1}} &\phi^2(y', h t)\,dy'P_b(t) dt\\
  &=
  \int_0^{2\ell/h}\int_{\mathbb{R}^{d-1}} \phi^2(y', h t)\,dy' P_b(t) dt\\
  &=
  \frac{b}{2}\int_{\mathbb{R}^{d-1}} \phi^2(y', 0)\,dy'
  - \int_{2\ell/h}^{\infty}\int_{\mathbb{R}^{d-1}} \phi^2(y', 0)\,dy' P_b(t) dt\\
  &\quad
  + 2\int_0^{2\ell/h}ht\int_{\mathbb{R}^{d-1}} \int_0^1\phi(y', h t s)\partial_{y_d}\phi(y', h t s)\,ds\,dy' P_b(t) dt\\
  &=
  \frac{b}{2}\int_{\mathbb{R}^{d-1}} \phi^2(y', 0)\,dy'
  + O(h\ell^{d-2}|{\log(\ell/h)}|)\,.
\end{align*}
Combined with~\eqref{eq: extraction of main term},~\eqref{eq: Upper tracebound halfspace}, and~\eqref{eq: lower trace bound halfspace} this completes the proof of Lemma~\ref{lem: halfspace asymptotics}.
\end{proof}

We are now ready to prove Theorem~\ref{thm: local asymptotics boundary}.
\begin{proof}[Proof of Theorem~\ref{thm: local asymptotics boundary}]
	By combining Lemma~\ref{lem: reduction to halfspace} and Lemma~\ref{lem: halfspace asymptotics} the claimed estimate follows from
	\begin{align*}
		\int_{\partial\Omega}\phi^2(x)&\biggl[b(x)-\inf_{y\in \Omega \cap B}b(y)\biggr]\,d\mathcal{H}^{d-1}(x)\\
		& \leq \int_{\partial\Omega}\phi^2(x)\biggl[\,\sup_{y\in \Omega \cap B}b(y)-\inf_{y\in \Omega \cap B}b(y)\biggr]\,d\mathcal{H}^{d-1}(x) \,,
	\end{align*}
	and the corresponding inequality for the $\sup$ and the fact that $\textup{supp}\,\phi \subseteq \overline{B} \subset \overline{B_{2\ell}(x)}$ for any $x\in \textup{supp}\,\phi$.
\end{proof}

\section{From local to global asymptotics}\label{sec: local to global}

In this section we prove our main result by piecing together the local asymptotics obtained above. The key ingredient is the following construction of a continuum partition of unity due to Solovej and Spitzer~\cite{MR2013804}.

Let
\begin{equation*}
  \ell(u) = \frac{1}{2}\max\{\mathrm{dist}(u, \Omega^c), 2\ell_0\}
\end{equation*}
with a small parameter $0<\ell_0$ to be determined. Note that $0<\ell\leq \max\{\tfrac{r_{in}(\Omega)}{2}, \ell_0\}$ and, since $|\nabla \textup{dist}(u, \Omega^c)|=1$ a.e., $\|\nabla \ell\|_{L^\infty}\leq \frac{1}{2}.$ Note also that $\mathrm{dist}(B_{\ell(u)}, \Omega^c))\leq 2\ell(u)$ if and only if $\mathrm{dist}(u, \partial \Omega)\leq 2\ell_0$ in which case $\ell(u)=\ell_0$. In particular, if $\mathrm{dist}(u, \Omega)>\ell_0$ then $B_{\ell(u)}(u)\cap\Omega = \emptyset$.

Fix a function $\phi\in C^\infty_0(\mathbb{R}^d)$ with $\textup{supp}\,\phi \subseteq \overline{B_1(0)}$ and $\|\phi\|_{L^2}=1$. By~\cite[Theorem~22]{MR2013804} (see also~\cite[Lemma 2.5]{FrankLarson}) the functions
$$
\phi_u(x) = \phi\left(\frac{x-u}{\ell(u)}\right)\, \sqrt{1+ \nabla\ell(u)\cdot\frac{x-y}{\ell(u)}} \,,
\qquad x\in\mathbb R^d \,,\ u\in\mathbb R^d \,,
$$
belong to $C_0^\infty(\mathbb R^d)$ with $\mathrm{supp}\, \phi_u \subseteq \overline{B_{\ell(u)}(u)}$, satisfy
\begin{equation}
\label{eq:phi_properties1}
\int_{\mathbb{R}^d} \phi_u(x)^2 \ell(u)^{-d}\, du = 1
\qquad\text{for all}\ x\in\mathbb{R}^d
\end{equation}
and, with a constant $C$ depending only on $d$,
\begin{equation*}
  \|\phi_u\|_{L^\infty}\leq \sqrt 2\, \|\phi\|_{L^\infty} \quad\text{and}\quad \|\nabla \phi_u\|_{L^\infty} \leq C\ell(u)^{-1} \|\nabla\phi\|_{L^\infty}
\quad\text{for all}\ u\in\mathbb{R}^d  \, .
\end{equation*}

The application to our problem here is summarized in the following lemma:

\begin{lem}\label{lem: localization Schrodinger}
  Let $\Omega, b, V$ be as in Theorem~\ref{thm: main result} and define $\ell, \{\phi_u\}_{u\in \mathbb{R}^d}$ as above. Then, for $0<\ell_0 \leq c(\Omega, b)$ and $0<h \leq K \ell_0$,
  \begin{align*}
     \Biggl|\textup{Tr}( H_{\Omega, b, V}(h))_- &- \int_{\mathbb{R}^d}\textup{Tr}(\phi_u H_{\Omega, b, V}(h)  \phi_u)_- \ell(u)^{-d}\,du\Biggr| \\
     &\leq C h^{-d+2} \int_{\mathrm{dist}(u, \Omega)\leq \ell_0}\Bigr(1+ h^2\|V_-\|^{1+d/2}_{L^{1+d/2}(\Omega \cap B_{\ell(u)}(u))}\Bigr)\ell(u)^{-2}\,du\,,
   \end{align*}
   where the constant $C$ depends only on $\Omega, b, K, \|\phi\|_{L^\infty}.$
\end{lem}

For the sake of brevity, we omit the proof of Lemma~\ref{lem: localization Schrodinger} and instead refer the reader to the proof of~\cite[Lemma 2.8]{FrankLarson}. Lemma~\ref{lem: localization Schrodinger} can be proved in the same manner but replacing the use of a local Berezin--Li--Yau inequality by an application of Lemma~\ref{lem: local HLT}.

With the above results in hand we are ready to prove Theorem~\ref{thm: main result}.

\begin{proof}[Proof of Theorem~\ref{thm: main result}]
  Set $\ell_0=h/\varepsilon_0$ with $0<h\leq \varepsilon_0 r_{in}(\Omega)/2$ for a parameter $\varepsilon_0 \in (0, 1]$ which will eventually tend to zero.

   We divide the set of $u\in\mathbb{R}^d$ such that $B_{\ell(u)}(u)\cap \Omega \neq \emptyset$ into two disjoint parts:
   \begin{equation}\label{eq: partition}
      \Omega_* = \{u \in \mathbb{R}^d: 2\ell_0<\delta_\Omega(u)\}\, \quad \mbox{and}\quad
      \Omega^* = \{u \in \mathbb{R}^d: -\ell_0< \delta_{\Omega}(u)\leq 2\ell_0\}\,,
    \end{equation}
    where $\delta_\Omega$ denotes the signed distance function to the boundary,
    $\delta_\Omega(y) = \textup{dist}(u, \Omega^c)-\textup{dist}(u, \Omega).$
    Note that for all $u\in \Omega^*$ we have $\ell(u)=\ell_0$.

    By Lemma~\ref{lem: localization Schrodinger} we need to understand the integral with respect to $u$ of the local traces $\textup{Tr}(\phi_u H_{\Omega, b, V}(h)\phi_u)_-$. Breaking the integral according to the partition~\eqref{eq: partition} we have
    \begin{align*}
      \int_{\mathbb{R}^d}\textup{Tr}(\phi_u H_{\Omega, b, V}(h)\phi_u)_- \ell(u)^{-d}\,du 
      &=
      \int_{\Omega_*}\textup{Tr}(\phi_u H_{\Omega, b, V}(h)\phi_u)_- \ell(u)^{-d}\,du\\
      &\quad
      +
      \int_{\Omega^*}\textup{Tr}(\phi_u H_{\Omega, b, V}(h)\phi_u)_- \ell_0^{-d}\,du\,.
    \end{align*}

    For the first term Lemma~\ref{lem: local asymptotics bulk} with $V_0(x)= \tfrac{(b(x)^2-1/4)_-}{\textup{dist}(x, \partial\Omega)^{2}}$, $V_1 = V_-(x)$ yields
    \begin{align*}
      &\int_{\Omega_*}\textup{Tr}(\phi_u H_{\Omega, b, V}(h)\phi_u)_- \ell(u)^{-d}\,du\\
      &\ \ \  =
      L_d h^{-d}\int_{\Omega_*}
        \int_\Omega \phi_u^2(x)\ell(u)^{-d}\,dxdu\\
        &\ \ \  + O(h^{-d+2})\!\!
        \int_{\Omega_*}\!\Bigl[
        \ell(u)^{-2}\bigl(1+ \|b\|_{L^\infty}^2\bigr) 
        + \|V_-\|_{L^{1+d/2}(B_{\ell(u)}(u))}^{1+d/2}+\ell(u)^{-d}\|V_+\|_{L^1(B_{\ell(u)}(u))}\Bigr]du
    \end{align*}
    where we used $\|V_0\|_{L^\infty} \leq \tfrac{C}{(\textup{dist}(u, \partial\Omega)-\ell(u))^{2}}\leq  C\ell(u)^{-2}$ and $\tfrac{(b(x)^2-1/4)_+}{\mathrm{dist}(x, \partial\Omega)^{2}}\leq C \|b\|_{L^\infty}^2 \ell(u)^{-2}$.

    \medskip
    For the integral over the boundary region $\Omega^*$ Theorem~\ref{thm: local asymptotics boundary}, for $\varepsilon_0, \ell_0, h$ sufficiently small, implies
    \begin{align*}
      &\int_{\Omega^*}\textup{Tr}(\phi_u H_{\Omega, b, V}(h)\phi_u)_- \ell_0^{-d}\,du\\
      &\quad=
      L_d h^{-d}\int_{\Omega^*}\int_\Omega \phi_u^2(x)\ell_0^{-d}\,dxdu - \frac{L_{d-1}}{2}h^{-d+1}\int_{\Omega^*}\int_{\partial\Omega}\phi_u^2(x)b(x)\ell_0^{-d}\,d\mathcal{H}^{d-1}(x)du\\
      &\qquad
      + O(h^{-d})|\Omega^*|(
        o_{\ell_0\to 0^+}(1)  + \varepsilon_0^{2}|{\log(\varepsilon_0)}|) + h^{-d+1}o_{\ell_0\to 0^+}(1)\\
        &\qquad  
        + O(h^{-d+2})\int_{\Omega^*}\Bigl[\|V_-\|^{1+d/2}_{L^{1+d/2}(B_{\ell(u)}(u))}+\ell_0^{-d}\|V_+\|_{L^1(B_{\ell(u)}(u))}\Bigr]du\,.
    \end{align*}
    Here we used the fact that $b$ satisfies~\eqref{eq: b regularity}.

    Combining the estimates for the contribution from the bulk and boundary region, using~\eqref{eq:phi_properties1}, and estimating the integrals of the norms of $V_-, V_+$, we find
    \begin{equation}\label{eq: integral of local traces asymptotics}
    \begin{aligned}
      &\int_{\mathbb{R}^d}\textup{Tr}(\phi_u H_{\Omega, b, V}(h)\phi_u)_- \ell(u)^{-d}\,du\\
      &\quad=
      L_d h^{-d}|\Omega| - \frac{L_{d-1}}{2}h^{-d+1}\int_{\partial\Omega}b(x)\,d\mathcal{H}^{d-1}(x)\\
      &\qquad 
      +
      O(h^{-d})|\Omega^*|(o_{\ell_0\to 0^+}(1)+\varepsilon_0^2|{\log(\varepsilon_0)}|)  + h^{-d+1}o_{\ell_0\to 0^+}(1)\\
      &\qquad 
      + O(h^{-d+2})\bigl(1+\|b\|_{L^\infty}^2\bigr)\int_{\Omega_*}\ell(u)^{-2}\,du 
      +O(h^{-d+2})\Bigl[\|V_-\|^{1+d/2}_{L^{1+d/2}(\Omega)}+ \|V_+\|_{L^1(\Omega)}\Bigr]\,.\hspace{-26pt}
    \end{aligned}
    \end{equation}
    
    By~\cite[eq.'s (4.6)--(4.8)]{FrankLarson}, $\int_{\Omega_*}\ell(u)^{-2}\,du \leq C \ell_0^{-1}$ and $|\Omega^*|\leq C \ell_0$ with $C$ depending only on $\Omega$. Thus by Lemma~\ref{lem: localization Schrodinger},~\eqref{eq: integral of local traces asymptotics}, and since $h^2/\ell(u)^{2}\leq \varepsilon_0^2$ we conclude that
    \begin{align*}
      h^{d-1}\biggl|\textup{Tr}(H_{\Omega, b, V}&(h))_- - L_d h^{-d}|\Omega| + \frac{L_{d-1}}{2}h^{-d+1}\int_{\partial\Omega}b(x)\,d\mathcal{H}^{d-1}(x)\biggr|\\
      &\leq
      \varepsilon_0^{-1}o_{h/\varepsilon_0\to 0^+}(1)+O(\varepsilon_0|{\log(\varepsilon_0)}|)  + o_{h/\varepsilon_0\to 0^+}(1)\\
      &\quad 
      + O(\varepsilon_0)\bigl(1+\|b\|_{L^\infty}^2\bigr)
      +O(h)\Bigl[\|V_-\|^{1+d/2}_{L^{1+d/2}(\Omega)}+ \|V_+\|_{L^1(\Omega)}\Bigr]\,.
    \end{align*}
    Letting first $h$ and then $\varepsilon_0$ tend to $0$ completes the proof of Theorem~\ref{thm: main result}. 
\end{proof}

\appendix

\section{Properties of \texorpdfstring{$P_\nu$}{Pv}}
\label{app: Properties of Pnu}

Our aim is to prove the following two lemmas.

\begin{lem}\label{lem: Pnu asymptotics}
  For $\nu \geq 0$ it holds that
  \begin{equation*}
    P_\nu(t) = \int_0^1(1-\xi^2)^{(d+1)/2}\Bigl(\frac{1}{\pi}- \xi t J_\nu(\xi t)^2\Bigr)\,d\xi = O(t^{-2})\quad \mbox{as }t \to \infty\,.
  \end{equation*}
  Moreover, the implicit constant is uniformly bounded for $\nu$ in compact subsets of\/ $[0, \infty)$.
\end{lem}

\begin{lem}\label{lem: Pnu integral identity}
  For any $\nu \geq 0$ we have the identity
  \begin{equation*}
    \int_0^\infty P_\nu(t)\,dt = \int_0^\infty \int_0^1(1-\xi^2)^{(d+1)/2}\Bigl(\frac{1}{\pi}- \xi t J_\nu(\xi t)^2\Bigr)\,d\xi dt= \frac{\nu}{2}\,.
  \end{equation*}
\end{lem}

We shall need the following asymptotic expansion for the Bessel function
\begin{equation}\label{eq: Bessel asymptotic form}
  J_\nu(t) = \Bigl(\frac{2}{\pi t}\Bigr)^{1/2}\biggl[\cos\Bigl(t- \frac{\nu\pi}{2}-\frac{\pi}{4}\Bigr)-\frac{4\nu^2-1}{8 t}\sin\Bigl(t- \frac{\nu\pi}{2}-\frac{\pi}{4}\Bigr) + O(t^{-2})\biggr]\,,
\end{equation}
where the implicit constant is uniformly bounded for $\nu$ in compact subsets of $[0, \infty)$ (see~\cite[Chapter 7]{Watson_BesselFunctions}). We shall also make use of the following identity 
\begin{equation}\label{eq: Bessel diff identity}
  x J_\nu(x)^2 = \frac{d}{dx}\Bigl[\frac{x^2}{2}J_\nu(x)^2+ \frac{x^2}{2}J_{\nu+1}(x)^2- \nu x J_\nu(x)J_{\nu +1}(x)\Bigr]\,,
\end{equation}
which is easily deduced from $J'_\nu(x)=\frac{1}{2}(J_{\nu-1}(x)-J_{\nu+1}(x))$ and the recursion formula $J_{\nu-1}(x)+J_{\nu+1}(x)= \frac{2\nu}{x}J_\nu(x)$.

\begin{proof}[Proof of Lemma~\ref{lem: Pnu asymptotics}]
  By an integration by parts,~\eqref{eq: Bessel diff identity}, and since $|J_\nu(x)|\leq 1$,
  \begin{align*}
    P_\nu(t) 
    =
    (d+1)\int_\delta^1 (1-\xi^2)^{(d-1)/2}\biggl[&\frac{\xi^2}{\pi} - \frac{t \xi^3}{2}J_\nu(\xi t)^2-\frac{t \xi^3}{2}J_{\nu+1}(\xi t)^2\\
    &\quad  + \nu \xi^2 J_\nu(\xi t)J_{\nu+1}(\xi t)\biggr]\,d\xi
    + O(t\delta^4+\delta^3)
  \end{align*}
  for any $0\leq \delta<1$. Provided $\delta t \gtrsim 1$,~\eqref{eq: Bessel asymptotic form} implies
  \begin{equation*}
    \frac{\xi^2}{\pi}- \frac{t \xi^3}{2}J_\nu(\xi t)^2-\frac{t \xi^3}{2}J_{\nu+1}(\xi t)^2+ \nu \xi^2 J_\nu(\xi t)J_{\nu+1}(\xi t)
    =
    \frac{\xi}{2\pi t}\cos(2 \xi t-\pi \nu) + O(t^{-2})\,,
  \end{equation*}
  with the implicit constant uniformly bounded for $\nu$ in compact subsets of $[0, \infty)$. Thus, we have arrived at
  \begin{align*}
    P_\nu(t)
    &=
    \frac{d+1}{2\pi t}\int_\delta^1 (1-\xi^2)^{(d-1)/2}\xi \cos(2\xi t-\pi \nu)\,d\xi + O(t^{-2}+ t\delta^4+\delta^3)\\
    &=
    \frac{d+1}{2\pi t}\int_0^1 (1-\xi^2)^{(d-1)/2}\xi \cos(2\xi t-\pi \nu)\,d\xi + O(t^{-2})\,,
  \end{align*}
  where we chose $\delta = O(t^{-1})$. An integration by parts yields
  \begin{equation*}
    \int_0^1 (1-\xi^2)^{(d-1)/2}\xi \cos(2\xi t-\pi \nu)\,d\xi
    =
    \frac{1}{2t}\int_0^1 (1-\xi^2)^{(d-3)/2}(d\xi^2- 1)\sin(2\xi t-\pi \nu)\,d\xi\,.
  \end{equation*}
  Since the integral on the right is bounded uniformly in $\nu$, this completes the proof.
\end{proof}

\begin{proof}[Proof of Lemma~\ref{lem: Pnu integral identity}]
For any $T>0$, by~\eqref{eq: Bessel diff identity}, Fubini's theorem, and a change of variables
\begin{align*}
  \int_0^T P_\nu(t)\,dt 
  &=
  \int_0^1 (1-\xi^2)^{(d+1)/2} \int_0^T\Bigl(\frac{1}{\pi}-\xi t J_{\nu}(\xi t)^2\Bigr)\,dtd\xi\\
  &=
  \frac{T}{2}(P_\nu(T)+P_{\nu+1}(T))+ \nu \int_0^T(1-s^2/T^2)^{(d+1)/2}J_\nu(s)J_{\nu+1}(s)\,ds\,.
\end{align*}
By Lemma~\ref{lem: Pnu asymptotics} only the remaining integral contributes as $T\to \infty$. By~\cite[p. 406]{Watson_BesselFunctions} and for $\nu>-1$, in the sense of an improper Riemann integral
\begin{equation*}
  \int_0^\infty J_\nu(s)J_{\nu+1}(s)\,ds = \frac{1}{2}\,.
\end{equation*}
The proof is completed by appealing to a simple Abelian theorem in Lemma~\ref{lem: Abelian lemma}.
\end{proof}

\begin{lem}\label{lem: Abelian lemma}
  If $f\in L^\infty(\mathbb{R}_+)$ and $\lim_{T\to \infty}\int_0^T f(t)\,dt = A$, then for all $\alpha >0$
  \begin{equation*}
    \lim_{T\to \infty}\int_0^T \Bigl(1- \frac{t^2}{T^2}\Bigr)^\alpha f(t)\,dt = A\,.
  \end{equation*}
\end{lem}

\begin{proof}
  By integration by parts and a change of variables,
  \begin{align*}
    \int_0^T \Bigl(1-\frac{t^2}{T^2}\Bigr)^\alpha f(t)\,dt
    &=
    \int_0^T \Bigl(-\frac{d}{dt}\Bigl(1- \frac{t^2}{T^2}\Bigr)^\alpha\Bigr)\int_0^t f(s)\,dsdt\\
    &=
    2\alpha \int_0^1 (1-\sigma^2)^{\alpha-1}\sigma \int_0^{\sigma T}f(s)\,dsd\sigma\,.
  \end{align*}
  By our assumptions there is a $S_0<\infty$ so that for $S\geq S_0$
  \begin{equation*}
    \biggl|\int_0^S f(s)\,ds\biggr| \leq |A|+1\,.
  \end{equation*}
  Since $f$ is bounded,
  \begin{equation*}
    \biggl|\int_0^S f(s)\,ds\biggr| \leq S\|f\|_{\infty}\,.
  \end{equation*}
  Thus, for all $\sigma, T$,
  \begin{equation*}
    \biggl|\int_0^{\sigma T}f(s)\,ds \biggr| \leq \max\{|A|+1, S_0\|f\|_{\infty}\}\,.
  \end{equation*}
  Since $\alpha >0$, the function $\sigma \mapsto (1-\sigma^2)^{\alpha-1}\sigma$ is integrable and by dominated convergence,
  \begin{equation*}
    \lim_{T\to \infty} 2\alpha \int_0^1 (1-\sigma^2)^{\alpha-1}\sigma \int_0^{\sigma T}f(s)\,dsd\sigma = 2\alpha A \int_0^1 (1-\sigma^2)^{\alpha-1}\sigma\, d\sigma = A\,.
  \end{equation*}
  This completes the proof of Lemma~\ref{lem: Abelian lemma}.
\end{proof}

% Optional acknowledgments should appear right before the bibliography.
\begin{acknowledgments}
U.S.~National Science Foundation grants DMS-1363432 and DMS-1954995 (R.L.F.) and Knut and Alice Wallenberg Foundation grant KAW~2018.0281 (S.L.) are acknowledged.
\end{acknowledgments}

\end{document}